\documentclass[12pt,reqno]{amsart}
\usepackage[top=2cm,bottom=2cm,right=2.5cm,left=2.5cm]{geometry}
\usepackage{amssymb}
\usepackage{amsmath, amsthm, amscd, amsfonts, amssymb, graphicx, color}
\usepackage[bookmarksnumbered, colorlinks, plainpages]{hyperref}

\usepackage{hyperref}

\newtheorem{theorem}{Theorem}[section]
\newtheorem{lemma}[theorem]{Lemma}
\newtheorem{proposition}[theorem]{Proposition}
\newtheorem{corollary}[theorem]{Corollary}
\theoremstyle{definition}
\newtheorem{definition}[theorem]{Definition}
\newtheorem{example}[theorem]{Example}

\theoremstyle{remark}
\newtheorem{remark}[theorem]{Remark}

\numberwithin{equation}{section}

\begin{document}
	
	\setcounter{page}{1}
	
	\title  [Continuous $bi$-$g$-frames for operators]{Continuous $bi$-$g$-frames for operators in Hilbert spaces}

	\author[A. Karara, M. Rossafi]{Abdelilah Karara$^{1}$ and Mohamed Rossafi$^{2*}$}
	
	\address{$^{1}$Department of Mathematics Faculty of Sciences, University of Ibn Tofail, B.P. 133, Kenitra, Morocco}
	\email{\textcolor[rgb]{0.00,0.00,0.84}{ abdelilah.karara.sm@gmail.com}}
	\address{$^{2}$Department of Mathematics Faculty of Sciences, Dhar El Mahraz University Sidi Mohamed Ben Abdellah, Fes, Morocco}
	\email{\textcolor[rgb]{0.00,0.00,0.84}{rossafimohamed@gmail.com}}
	
	
	
	\begin{abstract}
In this paper, we will introduce the new concepts of continuous  bi-$g-$frames and continuous $K$-bi-$g$-frame for Hilbert spaces. Then, we examine some characterizations  properties with the help of a biframe operator. Finally, we investigate several results about the stability of continuous bi-$g$-Bessel sequence and $K$-bi-$g$-frame are produced via the use of frame theory methods.
\end{abstract}
\maketitle

\baselineskip=12.4pt

\section{Introduction}
 
The notion of frames in Hilbert spaces was introduced by Duffin and Schaffer \cite{Duffin} in 1952 to research certain difficult nonharmonic Fourier series problems. Following the fundamental paper \cite{Daubechies} by Daubechies, Grossman, and Meyer, frame theory started to become popular, especially in the more specific context of Gabor frames and wavelet frames \cite{Gabor}. A sequence $\left\{\Phi_{\omega}\right\}_{\omega \in \Omega}$ in $\mathcal{H}$ is called a frame for $\mathcal{H}$ if there exist two constants $0<A \leq B<\infty$ such that
$$
A\left\| f \right\|^2 \leq \sum_{\omega \in \Omega}\vert\left\langle f, \Phi_{\omega}\right\rangle\vert^{2}\leq B\left\| f \right\|^2, \text { for all } f \in \mathcal{H} .
$$ For more detailed information on frame theory, readers are recommended to consult: \cite{Assila, Christensen, Daubechies2, Ghiati, Han, Massit, Rossafi}.

The concept of $K$-frames was introduced by Laura Găvruţa \cite{Gav}, serves as a tool for investigating atomic systems with respect to a bounded linear operator $K$ in a separable Hilbert space. A sequence $\left\{\Phi_{\omega}\right\}_{\omega \in \Omega}$ in $\mathcal{H}$ is called a $K$-frame for $\mathcal{H}$ if there exist two constants $0<A \leq B <\infty$ such that
$$
A\left\|K^{\ast} f\right\|^2 \leq \sum_{\omega \in \Omega}\vert\left\langle f, \Phi_{\omega}\right\rangle\vert^{2} \leq B\|f\|^2, \text { for all } f \in \mathcal{H} .
$$ 

The notion of $K$-frames generalize ordinary frames in that the lower frame bound is applicable only to elements within the range of $K$. After that Xiao et al. \cite{Xiao} introduced the concept of a K-g-frame, which is a more general framework than both g-frames and K-frames in Hilbert spaces.

The idea of pair frames, which refers to a pair of sequences in a Hilbert space, was first presented in \cite{Fer} by Azandaryani and Fereydooni. Parizi, Alijani and Dehghan \cite{MF} studied biframe, which is a generalization of controlled frame in Hilbert space. The concept of a frame is defined by a single sequence, but to define a biframe we will need two sequences. In fact, the concept of a biframe is a generalization of controlled frames and a special case of pair frames.

 In this paper, we will introduce the concept of continuous bi-$g$-frames and continuous $K$-bi-$g$-frame in Hilbert space and present some examples of this type of frame. Moreover, we investigate a characterization of continuous bi-$g$-frames and continuous $K$-bi-$g$-frame by using the biframe operator. Finally, in our exploration of biframes, we investigate  some results about the stability of continuous bi-$g$-Bessel sequence and continuous $K$-bi-$g$-frame are produced via the use of frame theory.

\section{Notation and preliminaries}
Throughout this paper, $\mathcal{H}$ represents a separable Hilbert space. The notation $\mathcal{B}(\mathcal{H},\mathcal{K})$ denotes the collection of all bounded linear operators from $\mathcal{H}$ to the Hilbert space $\mathcal{K}$. When $\mathcal{H} = \mathcal{K}$, this set is denoted simply as $\mathcal{B}(\mathcal{H})$. We will use $\mathcal{N}(\mathcal{T})$ and $\mathcal{R}(\mathcal{T})$ for the null  and range space of an operator $\mathcal{T}\in \mathcal{B}(\mathcal{H})$. Also $\mathrm{GL}(\mathcal{H})$ is the collection of all invertible, bounded linear operators acting on $\mathcal{H}$. $\left\{\mathcal{K}_{\omega}\right\}_{\omega \in \Omega}$ is a sequence of closed subspaces of $\mathcal{H}$. 

Certainly, let's begin with some preliminaries. Before diving into the details, let's briefly recall the definition of a continuous $g$-frame:
\begin{definition}
 We say that $\Phi=\left\{\Phi_\omega \in \mathcal{B}\left(\mathcal{H}, \mathcal{K}_\omega\right): \omega \in \Omega\right\}$ is a continuous generalized frame or simply a continuous $g$-frame with respect to $\left\{\mathcal{K}_\omega\right\}_{\omega \in \Omega}$ for $H$ if 
 \begin{enumerate}
 \item[•]For each $f \in \mathcal{H},\left\{\Phi_\omega f\right\}_{\omega \in \Omega}$ is strongly measurable,
 
 \item[•] there are two constants $0<A \leq B<\infty$ such that
$$
A\|f\|^2 \leq \int_{\Omega}\left\|\Phi_\omega f\right\|^2 d \mu(\omega) \leq B\|f\|^2, f \in H .
$$
 \end{enumerate}

We call $A, B$ lower and upper continuous $g$-frame bounds, respectively.
\begin{theorem}\cite{Abramovich}
$\mathcal{T} \in\mathcal{B}(\mathcal{H})$ is an injective and closed range operator if and only if there exists a constant $c>0$ such that $c\|f\|^2 \leq\|\mathcal{T}  f\|^2$, for all $f \in \mathcal{H}$
\end{theorem}
\begin{definition}\cite{Limaye} Let $\mathcal{H}$ be a Hilbert space, and suppose that $\mathcal{T} \in \mathcal{B}(\mathcal{H})$ has a closed range. Then there exists an operator $\mathcal{T}^{+} \in \mathcal{B}(\mathcal{H})$ for which
$$
N\left(\mathcal{T}^{+}\right)=\mathcal{R}(\mathcal{T})^{\perp}, \quad R\left(\mathcal{T}^{+}\right)=N(\mathcal{T})^{\perp}, \quad \mathcal{T} \mathcal{T}^{+} f=f, \quad f \in \mathcal{R}(\mathcal{T}) .
$$

We call the operator $\mathcal{T}^{+}$ the pseudo-inverse of $\mathcal{T}$. This operator is uniquely determined by these properties. In fact, if $\mathcal{T}$ is invertible, then we have $\mathcal{T}^{-1}=\mathcal{T}^{+}$.
\end{definition}
\end{definition}
\begin{theorem}\cite{Douglas} \label{Doglas th} 
 Let $\mathcal{H}$ be a Hilbert space and $\mathcal{T}_{1},\mathcal{T}_{2}\in\mathcal{B}(\mathcal{H})$. The following statements are equivalent:
 \begin{enumerate}
 \item $\mathcal{R}(\mathcal{T}_{1})\subset\mathcal{R}(\mathcal{T}_{2})$
 \item  $\mathcal{T}_{1} \mathcal{T}_{1}^{\ast} \leq \lambda^2 \mathcal{T}_{2}\mathcal{T}_{2}^{\ast}$ for some $\lambda \geq 0$;
 \item $\mathcal{T}_{1} = \mathcal{T}_{2} U$ for some $U\in\mathcal{B}(\mathcal{H})$.
 
 \end{enumerate}
\end{theorem}
\begin{lemma}\cite{Gaz} \label{Lemma2.5}
Let $\mathcal{T}: \mathcal{H} \rightarrow \mathcal{H}$ be a linear operator, and assume that there exist constants $\alpha, \beta \in[0 ; 1)$ such that
$$
\|\mathcal{T} f-f\| \leq \alpha\|f\|+\beta\|\mathcal{T} f\|, \forall f \in \mathcal{H} .
$$

Then $\mathcal{T} \in \mathcal{B}(\mathcal{H})$, and
$$
\dfrac{1-\alpha}{1+\beta}\|f\| \leq\|\mathcal{T} f\| \leq \dfrac{1+\alpha}{1-\beta}\|f\|,\quad \dfrac{1-\beta}{1+\alpha}\|f\| \leq\left\|\mathcal{T}^{-1} f\right\| \leq \dfrac{1+\beta}{1-\alpha}\|f\|, \forall f \in \mathcal{H} .
$$
\end{lemma}

\section{continuous bi-$g$-frames in Hilbert spaces}\label{s3}
In this section, we introduce the concept of a continuous bi-$g$-frame and subsequently establish some of its properties. But first we give the definition of bi-$g$-frame in Hilbert spaces. Throughout the rest of this part (sections \ref{s3},\ref{S4} and \ref{s5}), we denote: $$(\Phi, \Psi)=\left(\left\{\Phi_{\omega}\;:\;\Phi_{\omega}\in\mathcal{B}(\mathcal{H},\mathcal{K}_{\omega})\right\}_{\omega \in \Omega},\left\{\Psi_{\omega}\;:\;\Psi_{\omega}\in\mathcal{B}(\mathcal{H},\mathcal{K}_{\omega})\right\}_{\omega \in \Omega}\right)$$

\begin{definition}\label{c-bi-g-frame}
A pair $(\Phi, \Psi)$ of sequences is called a continuous bi-$g$-frame for $\mathcal{H}$ with respect to $\left\{\mathcal{K}_\omega\right\}_{\omega\in \Omega} $, if 
 \begin{enumerate}
 \item[•]for each $f \in \mathcal{H},\left\{\Phi_\omega f\right\}_{\omega \in \Omega}$ and $\left\{\Psi_\omega f\right\}_{\omega \in \Omega}$ are weakly-measurable.
 
 \item[•] there are two constants $0<A \leq B<\infty$ such that
$$
A\|f\|^2 \leq \int_{\Omega}\langle \Phi_\omega f,\Psi_\omega f \rangle d \mu(\omega) \leq B\|f\|^2, f \in H .
$$
 \end{enumerate}
The constants $A$ and $B$ are called continuous bi-$g$-frame bounds. If $A = B$, then it is called a tight continuous bi-$g$-frame and if $A = B = 1$, then it is called Parseval continuous bi-$g$-frame .
\end{definition}

\begin{remark}
According to Definition \ref{c-bi-g-frame}, the following statements are true for a sequence $\Phi=\left\{\Phi_{\omega}\;:\;\Phi_{\omega}\in\mathcal{B}(\mathcal{H},\mathcal{K}_{\omega})\right\}_{\omega \in \Omega}$ for $\mathcal{H}$ with respect to $\left\{\mathcal{K}_{\omega}\right\}_{\omega \in \Omega} $:
\begin{enumerate}
\item If $(\Phi, \Phi)$ is a continuous bi-$g$-frames for $\mathcal{H}$, then $\Phi$ is a continuous $g$-frame for $\mathcal{H}$.
\item If $(\Phi, C \Phi)$ is a continuous bi-$g$-frames for some $C \in \mathrm{GL}(\mathcal{H})$, then $\Phi$ is a continuous $C$-controlled $g$-frame for $\mathcal{H}$.
\item If $(C_{1} \Phi, C_{2} \Phi)$ is a continuous bi-$g$-frames for some $C_{1}$ and $C_{2}$ in $\mathrm{GL}(\mathcal{H})$, then $\Phi$ is a continuous  $(C_{1}, C_{2})$-controlled $g$-frame for $\mathcal{H}$.
\end{enumerate}
\end{remark}
\begin{example}
Let $\mathcal{H}$ be a separable Hilbert space and $(\left\{\alpha_\omega\right\}_{\omega \in \Omega},\left\{\beta_\omega\right\}_{\omega \in \Omega})$ be a biframe with bounds $A$ and $B$. Let $\Lambda_{\alpha_\omega}$ and $\Gamma_{\beta_\omega}$ be the functional induced by:
$$
\left\{\begin{array}{l}
\Phi_{\alpha_\omega}(f)=\left\langle f, \alpha_\omega\right\rangle \\
\Psi_{\beta_\omega}(f)=\left\langle \beta_\omega,f \right\rangle
\end{array}\right.
$$
Then for every $f\in \mathcal{H},$ we have
$$
A\|f\|^2 \leq \int_{\Omega}\langle \Phi_\omega f,\Psi_\omega f \rangle d \mu(\omega) = \int_{\Omega}\left\langle f, \alpha_\omega\right\rangle \left\langle \beta_\omega,f \right\rangle  d \mu(\omega) \leq B\|f\|^2, f \in \mathcal{H} .
$$
Hence $\left(\Phi_{\alpha}, \Psi_{\beta}\right)=\left(\left\{\Phi_{\alpha_\omega}\right\}_{\omega \in \Omega},\left\{\Psi_{\beta_\omega}\right\}_{\omega \in \Omega}\right)$ is a continuous bi- $g$-frame for $\mathcal{H}$ .
\end{example}
\begin{theorem}
  $(\Phi, \Psi)$ is a continuous bi-$g$-frames if and only if $(\Psi, \Phi)=\left(\left\{\Psi_{\omega}\right\}_{\omega \in \Omega},\left\{\Phi_{\omega}\right\}_{\omega \in \Omega}\right)$ is a continuous bi-$g$-frames.
\end{theorem}
\begin{proof}
Let $(\Phi, \Psi)$ be a continuous bi-$g$-frames with bounds $A$ and $B$. Then, for every $f \in \mathcal{H}$,
$$
A\left\|f\right\|^2\leq \int_{\Omega}\langle \Phi_\omega f,\Psi_\omega f \rangle d \mu(\omega) \leq B\left\| f \right\|^2.
$$

Now, we can write
$$
\begin{aligned}
\int_{\Omega}\langle \Phi_\omega f,\Psi_\omega f \rangle d \mu(\omega)&=\overline{\int_{\Omega}\langle \Phi_\omega f,\Psi_\omega f \rangle d \mu(\omega)}\\
&=\int_{\Omega}\overline{\langle \Phi_{\omega} f,\Psi_{\omega} f \rangle}d \mu(\omega)\\ &=\int_{\Omega}\langle  \Psi_{\omega}f,\Phi_{\omega} f \rangle d \mu(\omega).
\end{aligned}
$$
Therefore
$$
A\left\| f \right\|^2 \leq \int_{\Omega}\langle \Psi_\omega f,\Phi_\omega f \rangle d \mu(\omega) \leq B\left\| f \right\|^2 .
$$

This implies that, $(\Psi, \Phi)$ is a continuous bi-$g$-frames with bounds $A$ and $B$. The reverse of this statement can be proved similarly..
\end{proof}
\section{Continuous bi-$g$-frame Operator on Hilbert spaces } \label{S4}

Let $\Phi=\left\{\Phi_{\omega}\;:\;\Phi_{\omega}\in\mathcal{B}(\mathcal{H},\mathcal{K}_{\omega})\right\}_{\omega \in \Omega}$ and $\Psi=\left\{\Psi_{\omega}\;:\;\Psi_{\omega}\in\mathcal{B}(\mathcal{H},\mathcal{K}_{\omega})\right\}_{\omega \in \Omega}$ are two sequence for $\mathcal{H}$ with respect to $\mathcal{K}_{\omega}$. Assume that $(\Phi, \Psi)$ is a continuous bi-$g$-frame for $\mathcal{H}$, we define the continuous bi-$g$-frame operator $S_{\Phi,\Psi}$ as follows: 

$$S_{\Phi,\Psi}: \mathcal{H}\longrightarrow\mathcal{H}, \quad\quad\quad S_{\Phi,\Psi} f=\int_\Omega \Psi_\omega^* \Phi_\omega f d \mu(\omega) $$
\begin{proposition}
The operator $S_{\Phi,\Psi}$ is  selfadjoint, bounded, positive, invertible and $\left\|S_{\Phi,\Psi}^{-1}\right\| \leq \dfrac{1}{A}.$
\end{proposition}
\begin{proof}
For any $f, g \in \mathcal{H}$ we have :
$$
\begin{aligned}
\langle S_{\Phi,\Psi} f, g\rangle & =\left\langle\int_\Omega \Psi_\omega^* \Phi_\omega f d \mu(\omega), g\right\rangle\\ &=\int_\Omega\left\langle\Psi_\omega^* \Phi_\omega f, g\right\rangle d \mu(\omega) \\
& =\int_\Omega\left\langle f, \Phi_\omega^* \Psi_\omega g\right\rangle d \mu(\omega)\\ &=\left\langle f, \int_\Omega \Phi_\omega^* \Psi_\omega g d \mu(\omega)\right\rangle \\
& =\langle f, S_{\Psi,\Phi} g\rangle .
\end{aligned}
$$
Hence $S_{\Phi,\Psi}$ is  selfadjoint.
we have
$$
\begin{aligned}
\langle S_{\Phi,\Psi} f, f\rangle & =\left\langle\int_\Omega \Psi_\omega^* \Phi_\omega f d \mu(\omega), f\right\rangle\\ &=\int_\Omega\left\langle\Psi_\omega^* \Phi_\omega f, f\right\rangle d \mu(\omega)\\ &=\int_\Omega\left\langle \Phi_\omega f, \Psi f\right\rangle d \mu(\omega).
\end{aligned}
$$
Now we show that $S_{\Phi,\Psi}$ is a bounded operator
$$
\|S_{\Phi,\Psi}\|=\sup _{\|f\| \leq 1}\|\langle S_{\Phi,\Psi} f, f\rangle\|=\sup _{\|f\| \leq 1}\left\|\int_\Omega\left\langle\Phi_\omega f, \Psi_\omega f\right\rangle d \mu(\omega)\right\| \leq B .
$$
So
$$
A\langle f, f\rangle \leq\langle S_{\Phi,\Psi} f, f\rangle \leq B\langle f, f\rangle
$$
Therefore $$A I \leq S_{\Phi,\Psi} \leq B I,$$ Hence $S_{\Phi,\Psi}$ is a positive operator. Furthermore, $$0 \leq A^{-1} S_{\Phi,\Psi}-I \leq((B-A) / A) I.$$ So 
$$\left\|A^{-1} S_{\Phi,\Psi}-I\right\| \leq 1$$ which implies that $S_{\Phi,\Psi}$ is invertible.
Moreover we have 
$$A\|f\|^2 \leq\left\langle S_{\Lambda, \Gamma}(f), f\right\rangle \leq\left\|S_{\Phi,\Psi}(f)\right\|\|f\|, \quad(f \in \mathcal{H}).$$
Then
$$A\|f\| \leq\left\|S_{\Phi,\Psi}(f)\right\|, \quad(f \in \mathcal{H}).$$
Then
$$A\left\|S_{\Phi,\Psi}^{-1}(f)\right\| \leq\|f\|, \quad(f \in \mathcal{H}).$$
Hence
$$\left\|S_{\Phi,\Psi}^{-1}\right\| \leq \dfrac{1}{A}.$$
\end{proof}
The following theorem demonstrates that continuous bi-$g$-frames may also be used to accomplish the major feat of element reconstruction, which is one of the primary achievements of frames.
\begin{theorem}\label{th42}
Let $(\Phi, \Psi)$ be a continuous bi-$g$-frame for $\mathcal{H}$ with respect to $\left\{\mathcal{K}_\omega\right\}_{\omega \in \Omega}$ with continuous bi-$g$-frame operator $S_{\Phi, \Psi}$. Then, for every $f \in \mathcal{H}$, the following reconstruction formulas holds:
\begin{enumerate}
\item $f=\int_{\Omega} \Psi_\omega^* \Phi_\omega S_{\Phi, \Psi}^{-1} f d \mu(\omega)$
\item $f=\int_{\Omega}\left(\Psi_\omega S_{\Psi ,\Phi}^{-1}\right)^* \Phi_\omega f d \mu(\omega)$
\end{enumerate}
\end{theorem}
\begin{proof}
(1) For every $f\in \mathcal{H}$, we have $$f=S_{\Phi, \Psi} S_{\Phi, \Psi}^{-1} f=\int_{\Omega} \Psi_\omega^* \Phi_\omega S_{\Phi, \Psi}^{-1} f d \mu(\omega),
$$

(2)  For every $f\in \mathcal{H}$, we have $$f=S_{\Phi, \Psi}^{-1} S_{\Phi, \Psi} f=S_{\Phi, \Psi}^{-1} \int_{\Omega} \Psi_\omega^* \Phi_\omega f d \mu(\omega)=\int_{\Omega} S_{\Phi, \Psi}^{-1} \Psi_\omega^* \Phi_\omega f d \mu(\omega)=\int_{\Omega}\left(\Psi_\omega S_{\Psi ,\Phi}^{-1}\right)^* \Phi_\omega f d \mu(\omega).$$
\end{proof}
let $\tilde{\Phi}_\omega=\Phi_\omega S_{\Phi, \Psi}^{-1}$ and $\tilde{\Psi}_\omega=\Psi_\omega S_{\Psi, \Phi}^{-1}$. Therefore the equalities in the theorem \ref{th42} will be of this form

(1) $f=\int_{\Omega} \Psi_\omega^* \tilde{\Phi}_\omega f d \mu(\omega) $,

(2) $f=\int_{\Omega}\left(\tilde{\Psi}_\omega\right)^* \Phi_\omega f d \mu(\omega) $.

\begin{lemma}
 Let $(\Phi, \Psi)$ be a continuous bi-g-frame for $\mathcal{H}$ with respect to $\left\{\mathcal{K}_\omega\right\}_{\omega \in \Omega}$ with continuous bi-g-frame operator $S_{\Phi, \Psi}$. Then, $(\tilde{\Phi}, \tilde{\Psi})=\left(\left\{\tilde{\Phi}_\omega\right\}_{\omega \in \Omega},\left\{\tilde{\Psi}_\omega\right\}_{\omega \in \Omega}\right)$ is a bi-g-Bessel sequence for $\mathcal{H}$ with respect to $\left\{\mathcal{K}_\omega\right\}_{\omega \in \Omega}$ with continuous $g$-Bessel bound $\dfrac{1}{A}$.
\end{lemma}
\begin{proof}
For every $f \in \mathcal{H}$, we have
$$
\begin{aligned}
\int_{\Omega}\left\langle\tilde{\Phi}_\omega f, \tilde{\Psi}_\omega f\right\rangle d \mu(\omega) & =\int_{\Omega}\left\langle\Phi_\omega S_{\Phi, \Psi}^{-1} f, \Psi_\omega S_{\Psi, \Phi}^{-1} f\right\rangle d \mu(\omega) \\
& =\left\langle\int_{\Omega} \Psi_\omega^* \Phi_\omega S_{\Phi, \Psi}^{-1} f d \mu(\omega), S_{\Psi, \Phi}^{-1} f\right\rangle  \\
& =\left\langle S_{\Phi, \Psi} S_{\Phi, \Psi}^{-1} f, S_{\Psi, \Phi}^{-1} f\right\rangle  \\
& =\left\langle f, S_{\Psi, \Phi}^{-1} f\right\rangle \\
& \leq \dfrac{1}{A}\|f\|^2
\end{aligned}
$$
\end{proof}

\section{Stability of continuous bi-$g$-Bessel sequence for Hilbert spaces}\label{s5}
\begin{theorem} \label{th51}
Let $\Phi=\left\{\Phi_{\omega}\;:\;\Phi_{\omega}\in\mathcal{B}(\mathcal{H},\mathcal{K}_{\omega})\right\}_{\omega \in \Omega}$ and $\Psi=\left\{\Psi_{\omega}\;:\;\Psi_{\omega}\in\mathcal{B}(\mathcal{H},\mathcal{K}_{\omega})\right\}_{\omega \in \Omega}$ are two $g$-Bessel sequences with bounds $B_{\Phi}$, $B_{\Psi}$ respectively. Assume that $(\Phi, \Psi)$ be a continuous bi-$g$-fra for $\mathcal{H}$ with respect to $\left\{\mathcal{K}_{\omega}\right\}_{\omega \in \Omega}$ with bound $B$ and $\left(\left\{\Lambda_{\omega}\right\}_{\omega \in \Omega},\left\{\Gamma_{\omega}\right\}_{\omega \in \Omega}\right)$ be a pair of sequences for $\mathcal{H}$ with respect to $\left\{\mathcal{K}_{\omega}\right\}_{\omega \in \Omega}$. If there exist constants $\alpha, \beta, \gamma \in[0,1)$ such that $\max \left\{\alpha+\gamma , \beta\right\}<1$ and
$$
\begin{aligned}
\left\|\int_{\Omega}\left(\Psi_{\omega}^* \Phi_{\omega}-\Gamma_{\omega}^* \Lambda_{\omega}\right) f\right\| 
\leq & \alpha\left\|\int_{\Omega} \Psi_{\omega}^* \Phi_{\omega} f d \mu(\omega)\right\|+\beta\left\|\int_{\Omega} \Gamma_{\omega}^* \Lambda_{\omega} f d \mu(\omega)\right\|+\gamma\left\| f\right\|.
\end{aligned}
$$
Then  $\left(\left\{\Lambda_{\omega}\right\}_{\omega \in \Omega},\left\{\Gamma_{\omega}\right\}_{\omega \in \Omega}\right)$ is a continuous bi-$g$-Bessel for $\mathcal{H}$ with respect to $\left\{\mathcal{K}_{\omega}\right\}_{\omega \in \Omega}$ with bound
$$
 \dfrac{\left(1+\alpha\right) \sqrt{B_{\Phi}B_{\Psi}}+\gamma}{1-\beta}.
$$
\end{theorem}
\begin{proof} For any $f \in \mathcal{H}$, we have
$$
\begin{aligned}
\left\|\int_{\Omega} \Gamma_{\omega}^* \Lambda_{\omega} f d \mu(\omega)\right\| & \leq\left\|\int_{\Omega}\left(\Gamma_{\omega}^* \Lambda_{\omega}-\Psi_{\omega}^* \Phi_{\omega}\right) f d \mu(\omega)\right\|+\left\|\int_{\Omega} \Psi_{\omega}^* \Phi_{\omega} f d \mu(\omega)\right\| \\
& \leq\left(1+\alpha\right)\left\|\int_{\Omega} \Psi_{\omega}^* \Phi_{\omega} f d \mu(\omega)\right\|+\beta\left\|\int_{\Omega} \Gamma_{\omega}^* \Lambda_{\omega} f d \mu(\omega)\right\|+\gamma \left\| f d \mu(\omega)\right\| .
\end{aligned}
$$
Then
$$
\left\|\int_{\Omega} \Gamma_{\omega}^* \Lambda_{\omega} f d \mu(\omega)\right\| \leq \dfrac{1+\alpha}{1-\beta}\left\|\int_{\Omega} \Psi_{\omega}^* \Phi_{\omega} f d \mu(\omega)\right\|+\dfrac{\gamma}{1-\beta}\left\| f\right\| .
$$

Since
$$
\begin{aligned}
\left\|\int_{\Omega} \Psi_{\omega}^* \Phi_{\omega} f d \mu(\omega)\right\| & =\sup _{\|g\|=1}\left|\left\langle \int_{\Omega}\Psi_{\omega}^* \Phi_{\omega} f d \mu(\omega), g\right\rangle\right|\\ 
&=\sup _{\|g\|=1}\left|\int_{\Omega} \left\langle\Psi_{\omega}^* \Phi_{\omega} f , g\right\rangle d \mu(\omega)\right| \\
&=\sup _{\|g\|=1}\left|\int_{\Omega}\left\langle \Phi_{\omega} f, \Psi_{\omega} g\right\rangle d \mu(\omega)\right| \\
& \leq\left(\int_{\Omega}\left\|\Phi_{\omega} f\right\|^2\right)^{\dfrac{1}{2}} \sup _{\|g\|=1}\left(\int_{\Omega}\left\|\Psi_{\omega} g\right\|^2\right)^{\dfrac{1}{2}} \\
& \leq \sqrt{B_{\Phi}B_{\Psi}} \left\| f\right\|.
\end{aligned}
$$
 Hence, for all $f \in \mathcal{H}$, we have
$$
\left\|\int_{\Omega} \Gamma_{\omega}^* \Lambda_{\omega} f  d \mu(\omega)\right\| \leq \dfrac{\left(1+\alpha\right) \sqrt{B_{\Phi}B_{\Psi}}}{1-\beta}\left\| f\right\| +\dfrac{\gamma}{1-\beta}\left\| f\right\|= \dfrac{\left(1+\alpha\right) \sqrt{B_{\Phi}B_{\Psi}}+\gamma}{1-\beta} \|f\| .
$$
We considere
$$
\mathcal{M} :\mathcal{H} \rightarrow \mathcal{H}, \quad \mathcal{M} f=\int_{\Omega} \Gamma_{\omega}^* \Lambda_{\omega} f d \mu(\omega),\quad (f \in \mathcal{H}) .
$$
 $\mathcal{M}$ is well-defined, bounded and $$\|\mathcal{M}\| \leq \dfrac{\left(1+\alpha\right) \sqrt{B_{\Phi}B_{\Psi}}+\gamma}{1-\beta}.$$ For every $f \in \mathcal{H}$, we have 
\begin{equation} \label{eq51}
\langle \mathcal{M} f, f\rangle=\langle\int_{\Omega} \Gamma_{\omega}^* \Lambda_{\omega} f d \mu(\omega),f\rangle=\int_{\Omega}\langle \Lambda_{\omega} f,\Gamma_{\omega} f\rangle d \mu(\omega)\leq\|\mathcal{M}\|\|f\|^2
\end{equation} 
It implies that $\left(\left\{\Lambda_{\omega}\right\}_{\omega \in \Omega},\left\{\Gamma_{\omega}\right\}_{\omega \in \Omega}\right)$ is a continuous bi-$g$-Bessel sequence for $\mathcal{H}$ with respect to $\left\{\mathcal{K}_{\omega}\right\}_{\omega \in \Omega}$
\end{proof}
\begin{corollary}
 Let $\Phi=\left\{\Phi_{\omega}\;:\;\Phi_{\omega}\in\mathcal{B}(\mathcal{H},\mathcal{K}_{\omega})\right\}_{\omega \in \Omega}$ and $\Psi=\left\{\Psi_{\omega}\;:\;\Psi_{\omega}\in\mathcal{B}(\mathcal{H},\mathcal{K}_{\omega})\right\}_{\omega \in \Omega}$ are two $g$-Bessel sequences with bounds $B_{\Phi}$, $B_{\Psi}$ respectively. Assume that $(\Phi, \Psi)$ be a continuous bi-$g$-frame for $\mathcal{H}$ with respect to $\left\{\mathcal{K}_{\omega}\right\}_{\omega \in \Omega}$ with bounds  $A$ and $B$ and $\left(\left\{\Lambda_{\omega}\right\}_{\omega \in \Omega},\left\{\Gamma_{\omega}\right\}_{\omega \in \Omega}\right)$ be a pair of sequences for $\mathcal{H}$ with respect to $\left\{\mathcal{K}_{\omega}\right\}_{\omega \in \Omega}$.  If there exists constant $0<D<A$ such that
$$
\left\|\int_{\Omega}\left(\Psi_{\omega}^* \Phi_{\omega} -\Psi_\omega^* \Phi_\omega\right) f\right\| \leq D\left\| f\right\|, \forall f \in \mathcal{H},
$$
then $\left(\left\{\Lambda_{\omega}\right\}_{\omega \in \Omega},\left\{\Gamma_{\omega}\right\}_{\omega \in \Omega}\right)$ is a continuous bi-$g$-Bessel for $H$ with respect to $\left\{\mathcal{K}_\omega\right\}_{\omega \in \Omega}$ with bound $\left(\sqrt{B_{\Phi}B_{\Psi}}+D\sqrt{\dfrac{B}{A}}\right)$.
\end{corollary}
\begin{proof} For any $f \in \mathcal{H}$, we have  $$\left\| f\right\| \leq  \dfrac{1}{\sqrt{A}}\left(\int_{\Omega}\langle \Phi_\omega f,\Psi_\omega f \rangle d \mu(\omega)\right) ^{\dfrac{1}{2}}$$
Then
$$
\begin{aligned}
\left\|\int_{\Omega}\left(\Psi_{\omega}^* \Phi_{\omega} -\Psi_\omega^* \Phi_\omega\right) f\right\| &\leq D\left\|f\right\|\\ &\leq  \dfrac{1}{\sqrt{A}}\left(\int_{\Omega}\langle \Phi_\omega f,\Psi_\omega f \rangle d \mu(\omega)\right) ^{\dfrac{1}{2}}\\
&\leq  D\sqrt{\dfrac{B}{A}}\|f\|.
\end{aligned}
$$

By letting $\alpha, \beta=0, \gamma= D\sqrt{\dfrac{B}{A}}$ in Theorem \ref{th51}, $\left(\left\{\Lambda_{\omega}\right\}_{\omega \in \Omega},\left\{\Gamma_{\omega}\right\}_{\omega \in \Omega}\right)$ is a continuous bi-$g$-Bessel for $\mathcal{H}$ with respect to $\left\{\mathcal{K}_{\omega}\right\}_{\omega \in \Omega}$ with bound $\left(\sqrt{B_{\Phi}B_{\Psi}}+D\sqrt{\dfrac{B}{A}}\right)$.
\end{proof}
\begin{theorem} \label{th53}
 Let $\Phi=\left\{\Phi_{\omega}\;:\;\Phi_{\omega}\in\mathcal{B}(\mathcal{H},\mathcal{K}_{\omega})\right\}_{\omega \in \Omega}$ and $\Psi=\left\{\Psi_{\omega}\;:\;\Psi_{\omega}\in\mathcal{B}(\mathcal{H},\mathcal{K}_{\omega})\right\}_{\omega \in \Omega}$ are two $g$-Bessel sequences with bounds $B_{\Phi}$, $B_{\Psi}$ respectively. Assume that $(\Phi, \Psi)$ be a continuous bi-$g$-frame for $\mathcal{H}$ with respect to $\left\{\mathcal{K}_{\omega}\right\}_{\omega \in \Omega}$ with bounds $A$ and $B$ and $\left(\left\{\Lambda_{\omega}\right\}_{\omega \in \Omega},\left\{\Gamma_{\omega}\right\}_{\omega \in \Omega}\right)$ be a pair of sequences for $\mathcal{H}$ with respect to $\left\{\mathcal{K}_{\omega}\right\}_{\omega \in \Omega}$. If there exist constants $\alpha, \beta, \gamma \in[0,1)$ such that $\max \left\{\alpha+\gamma\sqrt{\dfrac{ B }{A}} , \beta\right\}<1$ and
$$
\begin{aligned}
\left\|\int_{\Omega}\left(\Psi_{\omega}^* \Phi_{\omega}-\Gamma_{\omega}^* \Lambda_{\omega}\right) f\right\| 
\leq & \alpha\left\|\int_{\Omega} \Psi_{\omega}^* \Phi_{\omega} f\right\|+\beta\left\|\int_{\Omega} \Gamma_{\omega}^* \Lambda_{\omega} f\right\|+\gamma\left\| f\right\|.
\end{aligned}
$$
Then $\left(\left\{\Lambda_{\omega}\right\}_{\omega \in \Omega},\left\{\Gamma_{\omega}\right\}_{\omega \in \Omega}\right)$ is a continuous bi-$g$-Bessel for $\mathcal{H}$ with respect to $\left\{\mathcal{K}_{\omega}\right\}_{\omega \in \Omega}$ with bound
$$
 \dfrac{\left[\left(1+\alpha\right) \sqrt{B_{\Phi}B_{\Psi}}+\gamma\sqrt{\dfrac{ B }{A}}\right] }{1-\beta}.
$$
\end{theorem}
\begin{proof}
The proof is analogous to that of Theorem \ref{th51}.
\end{proof}
\section{continuous $K$-bi-$g$-frames in Hilbert spaces}\label{s6}
In this section, we introduce the concept of a continuous $K$-bi-$g$-frame and subsequently establish some of its properties. But first we give the definition of continuous $K$-bi-$g$-frame in Hilbert spaces. Throughout the rest of this part (sections \ref{s6},\ref{s7} and \ref{s8}), we denote: $$(\Phi, \Psi)_{K}=\left(\left\{\Phi_{\omega}\;:\;\Phi_{\omega}\in\mathcal{B}(\mathcal{H},\mathcal{K}_{\omega})\right\}_{\omega \in \Omega},\left\{\Psi_{\omega}\;:\;\Psi_{\omega}\in\mathcal{B}(\mathcal{H},\mathcal{K}_{\omega})\right\}_{\omega \in \Omega}\right)$$

\begin{definition}\label{K-bi-g-frame}
Let $K \in \mathcal{B}(\mathcal{H})$. A pair $(\Phi, \Psi)_{K}$ of sequences is called a continuous $K$-bi-$g$-frame for $\mathcal{H}$  with respect to $\left\{\mathcal{K}_{\omega}\right\}_{\omega \in \Omega} $ , if there exist two constants $0<A \leq B<\infty$ such that
$$
A\left\|K^{\ast} f\right\|^2 \leq \int_{\Omega}\langle \Phi_{\omega} f,\Psi_{\omega} f \rangle d \mu(\omega) \leq B\|f\|^2, \text { for all } f \in \mathcal{H} .
$$
\end{definition}
The numbers $A$ and $B$ are called respectively the lower and upper bounds for the continuous $K$-bi-$g$-frames $(\Phi, \Psi)_{K}$ respectively. 
 If $K$ is equal to $\mathcal{I}_{\mathcal{H}}$, the identity operator on $\mathcal{H}$, then continuous $K$-bi-$g$-frames is continuous bi-$g$-frames.

\begin{remark}
According to Definition \ref{K-bi-g-frame}, the following statements are true for a sequence $\Phi=\left\{\Phi_{\omega}\;:\;\Phi_{\omega}\in\mathcal{B}(\mathcal{H},\mathcal{K}_{\omega})\right\}_{\omega \in \Omega}$ for $\mathcal{H}$ with respect to $\left\{\mathcal{K}_{\omega}\right\}_{\in I} $:
\begin{enumerate}
\item If $(\Phi, \Phi)$ is a continuous $K$-bi-$g$-frames for $\mathcal{H}$, then $\Phi$ is a continuous $K$-$g$-frame for $\mathcal{H}$.
\item If $(\Phi, C \Phi)$ is a continuous $K$-bi-$g$-frames for some $C \in \mathrm{GL}(\mathcal{H})$, then $\Phi$ is a $C$-controlled continuous $K$-$g$-frame for $\mathcal{H}$.
\item If $(C_{1} \Phi, C_{2} \Phi)$ is a continuous $K$-bi-$g$-frames for some $C_{1}$ and $C_{2}$ in $\mathrm{GL}(\mathcal{H})$, then $\Phi$ is a $(C_{1}, C_{2})$-controlled continuous $K$-$g$-frame for $\mathcal{H}$.
\end{enumerate}
\end{remark}
\begin{example} 
 Let $\left\{e_1, e_2, e_3\right\}$ be an standard orthonormal basis for $\mathcal{H}$. Assume that $\left\{\Omega_1, \Omega_2, \Omega_3\right\}$ is a partition of $\Omega$ where $\mu\left(\Omega_1\right), \mu\left(\Omega_2\right), \mu\left(\Omega_3\right)>0$. Define
 $$\begin{aligned} & \mathcal{X}: \Omega \rightarrow H \quad \text { by } \quad \mathcal{X}(\omega)= \begin{cases}\sqrt{\dfrac{2 }{\mu\left(\Omega_1\right)}}e_1 & \text { if } \omega \in \Omega_1 \\ \sqrt{\dfrac{3 }{\mu\left(\Omega_2\right)}}e_2 & \text { if } \omega \in \Omega_2 \\ \sqrt{\dfrac{3 }{\mu\left(\Omega_3\right)}}e_3 & \text { if } \omega \in \Omega_3\end{cases} \\ & \mathcal{Y}: \Omega \rightarrow H \quad \text { by } \quad \mathcal{Y}(\omega)= \begin{cases}\sqrt{\dfrac{8 }{\mu\left(\Omega_1\right)}}e_1 & \text { if } \omega \in \Omega_1 \\ \sqrt{\dfrac{3 }{\mu\left(\Omega_2\right)}}e_2 & \text { if } \omega \in \Omega_2 \\ 2\sqrt{\dfrac{3 }{\mu\left(\Omega_3\right)}}e_3 & \text { if } \omega \in \Omega_3\end{cases} \\ & \end{aligned}$$

It is easy to verify that for all $f \in H, \omega \mapsto\langle f, \mathcal{X}(\omega)\rangle$ and $\omega \mapsto\langle f, \mathcal{Y}(\omega)\rangle$ are measurable functions on $\Omega$. Define $K: H \rightarrow H$ by $K e_1=e_1, K e_2=e_3$
and $K e_3=e_2$. Then $K^* e_1=e_1, K^* e_2=e_3, K^* e_3=e_2$ and it is easy to verify that $\left\|K^* f\right\|^2=\|f\|^2$. Now, for $f \in H$, we have

$$
\begin{aligned}
& \int_{\Omega}\langle f, \mathcal{X}(\omega)\rangle\langle \mathcal{Y}(\omega), f\rangle d \mu(\omega) \\
& =\int_{\Omega_1}\left\langle f, \sqrt{\dfrac{2 }{\mu\left(\Omega_1\right)}}e_1\right\rangle\left\langle\sqrt{\dfrac{8 }{\mu\left(\Omega_1\right)}}e_1, f\right\rangle d \mu (\omega)\\
& +\int_{\Omega_2}\left\langle f,\sqrt{\dfrac{3 }{\mu\left(\Omega_1\right)}}e_1\right\rangle\left\langle\sqrt{\dfrac{3 }{\mu\left(\Omega_1\right)}}e_1, f\right\rangle d \mu(\omega) \\
& +\int_{\Omega_3}\left\langle f, \sqrt{\dfrac{3 }{\mu\left(\Omega_1\right)}}e_1\right\rangle\left\langle 2\sqrt{\dfrac{3 }{\mu\left(\Omega_1\right)}}e_1, f\right\rangle d \mu \\
& =4\left|\left\langle f, e_1\right\rangle\right|^2+3\left|\left\langle f, e_2\right\rangle\right|^2+6\left|\left\langle f, e_3\right\rangle\right|^2 \\
& =2\|f\|^2+2\left|\left\langle f, e_1\right\rangle\right|^2+3\left|\left\langle f, e_2\right\rangle\right|^2 .
\end{aligned}
$$

Hence, for every $f \in \mathcal{H}$, we obtain
$$
2\left\|K^* f\right\|^2 \leq \int_{\Omega}\langle f, \mathcal{X}(\omega)\rangle\langle \mathcal{Y}(\omega), f\rangle d \mu(\omega) \leq 3\|f\|^2 .
$$

Therefore, $(\mathcal{X}, \mathcal{Y})$ is a continuous $K$-biframe.

Now,  For each $\omega \in \Omega$, defines the operators 

 $$\begin{aligned} & \Phi_\omega: \mathcal{H} \longrightarrow \left\{\mathcal{K}_{\omega}\right\}_{\omega \in \Omega} \quad \text { by } \quad \Phi_\omega(f)= \begin{cases}\dfrac{1}{\sqrt{\mu\left(\Omega_1\right)}}\left\langle f, \mathcal{X}(\omega)\right\rangle \mathcal{I}(\omega) & \text { if } \omega \in \Omega_1 \\\dfrac{1}{\sqrt{\mu\left(\Omega_2\right)}}\left\langle f, \mathcal{X}(\omega)\right\rangle \mathcal{I}(\omega) & \text { if } \omega \in \Omega_2 \\\dfrac{1}{\sqrt{\mu\left(\Omega_3\right)}}\left\langle f, \mathcal{X}(\omega)\right\rangle \mathcal{I}(\omega) & \text { if } \omega \in \Omega_3\end{cases} \\
  And \;\;&\Psi_\omega: \mathcal{H} \longrightarrow \left\{\mathcal{K}_{\omega}\right\}_{\omega \in \Omega} \quad \text { by } \quad \Psi_\omega(f)= 
  \begin{cases}\dfrac{1}{\sqrt{\mu\left(\Omega_1\right)}}\left\langle \mathcal{Y}(\omega),f \right\rangle \mathcal{I}(\omega) & \text { if } \omega \in \Omega_1 \\\dfrac{1}{\sqrt{\mu\left(\Omega_2\right)}}\left\langle \mathcal{Y}(\omega),f\right\rangle \mathcal{I}(\omega) & \text { if } \omega \in \Omega_2 \\\dfrac{1}{\sqrt{\mu\left(\Omega_3\right)}}\left\langle \mathcal{Y}(\omega),f\right\rangle \mathcal{I}(\omega) & \text { if } \omega \in \Omega_3\end{cases} \\ & \end{aligned}$$

Where  $\left\|\mathcal{I}(\omega)\right\|=1$. For every $f \in \mathcal{H}$, we have
$$
\int_{\Omega}\langle \Phi_{\omega} f,\Psi_{\omega} f \rangle d \mu(\omega)=\int_{\Omega}\langle f, \mathcal{X}(\omega)\rangle\langle \mathcal{Y}(\omega), f\rangle d \mu(\omega).
$$
Which implies that

$$2\left\|K^* f\right\|^2 \leq\int_{\Omega}\langle \Phi_{\omega} f,\Psi_{\omega} f \rangle d \mu(\omega) \leq 3\|f\|^2$$
Hence, $(\Phi, \Psi)_K$ is a continuous $K$-bi-$g$-frame for $\mathcal{H}$  with respect to $\left\{\mathcal{K}_{\omega}\right\}_{\omega \in \Omega} $.
\end{example}

\begin{definition}\label{Def Tight}
Let $K \in \mathcal{B}(\mathcal{H})$. A pair  $(\Phi, \Psi)_{K}$ of sequences in $\mathcal{H}$ is said to be a tight continuous $K$-bi-$g$-frames with bound $A$ if
$$
A\left\|K^* f\right\|^2=\int_{\Omega}\langle \Phi_{\omega} f,\Psi_{\omega} f \rangle d \mu(\omega) , \text { for all } f \in \mathcal{H} .
$$

When $A=1$, it is called a Parseval continuous $K$-bi-$g$-frames.
\end{definition}

\begin{theorem}
  $(\Phi, \Psi)_{K}$ is a continuous $K$-bi-$g$-frames if and only if $(\Psi, \Phi)_{K}=\left(\left\{\Psi_{\omega}\right\}_{\omega \in \Omega},\left\{\Phi_{\omega}\right\}_{\omega \in \Omega}\right)$ is a continuous $K$-bi-$g$-frames.
\end{theorem}
\begin{proof}
Let $(\Phi, \Psi)_{K}$ be a continuous $K$-bi-$g$-frames with bounds $A$ and $B$. Then, for every $f \in \mathcal{H}$,
$$
A\left\|K^{\ast} f\right\|^2\leq \int_{\Omega}\langle \Phi_{\omega} f,\Psi_{\omega} f \rangle d \mu(\omega) \leq B\left\| f \right\|^2.
$$

Now, we can write
$$
\begin{aligned}
\int_{\Omega}\langle \Phi_{\omega} f,\Psi_{\omega} f \rangle d \mu(\omega)&=\overline{\int_{\Omega}\langle \Phi_{\omega} f,\Psi_{\omega} f \rangle d \mu(\omega)}\\
&=\int_{\Omega}\overline{\langle \Phi_{\omega} f,\Psi_{\omega} f \rangle }d \mu(\omega)\\ &=\int_{\Omega}\langle  \Psi_{\omega}f,\Phi_{\omega} f \rangle d \mu(\omega) .
\end{aligned}
$$
Hence
$$
A\left\|K^{\ast} f\right\|^2 \leq \int_{\Omega}\langle  \Psi_{\omega}f,\Phi_{\omega} f \rangle d \mu(\omega) \leq B\left\| f \right\|^2 .
$$

This implies that, $(\Psi, \Phi)_{K}$ is a continuous $K$-bi-$g$-frames with bounds $A$ and $B$. The reverse of this statement can be proved similarly..
\end{proof}

\begin{theorem}
Let $K_1, K_2 \in \mathcal{B}(\mathcal{H})$. If  $(\Phi, \Psi)_{K}$ is an continuous K$_{j}$-bi-$g$-frame for $j\in\lbrace 1,\;2\rbrace$  and $\alpha_{1}, \alpha_{2}$ are scalars. Then the following holds: 
\begin{enumerate}
\item $(\Phi, \Psi)_{K}$ is continuous $(\alpha_{1} K_1+\alpha_{2} K_2)$-bi-$g$-frame
\item $(\Phi, \Psi)_{K}$ is continuous $K_1 K_2$-bi-$g$-frame.
\end{enumerate}
\end{theorem}
\begin{proof}
(1) Let $(\Phi, \Psi)_{K}$
  be $K_{1}$-bi-$g$-frame and $K_{2}$-bi-$g$-frame. Then for $j=1$, there exist two constants $0<A \leq B<\infty$ such that
$$
A\left\|K_1^{\ast} f\right\|^2 \leq \int_{\Omega}\langle \Phi_{\omega} f,\Psi_{\omega} f \rangle d \mu(\omega) \leq B\|f\|^2, \text { for all } f \in \mathcal{H} .
$$
And for $j=2$, there exist two constants $0<C \leq D<\infty$ such that
$$
C\left\|K_2^{\ast} f\right\|^2 \leq \int_{\Omega}\langle \Phi_{\omega} f,\Psi_{\omega} f \rangle d \mu(\omega) \leq D\|f\|^2, \text { for all } f \in \mathcal{H} .
$$
Now, we can write
$$
\begin{aligned}
\left\|\left(\alpha_{1} K_1+\alpha_{2} K_2\right)^{\ast} f\right\|^2 & \leq|\alpha_{1}|^2\left\|K_1^{\ast} f\right\|^2+|\alpha_{2}|^2\left\|K_2^{\ast} f\right\|^2 \\
& \leq|\alpha_{1}|^2\left(\dfrac{1}{A} \int_{\Omega}\langle \Phi_{\omega} f,\Psi_{\omega} f \rangle d \mu(\omega)\right)+|\alpha_{2}|^2\left(\dfrac{1}{C} \int_{\Omega}\langle \Phi_{\omega} f,\Psi_{\omega} f \rangle d \mu(\omega)\right) \\
& =\left(\dfrac{|\alpha_{1}|^2}{A}+\dfrac{|\alpha_{2}|^2}{C}\right) \int_{\Omega}\langle \Phi_{\omega} f,\Psi_{\omega} f \rangle d \mu(\omega) .
\end{aligned}
$$

It follows that
$$
\left(\dfrac{A C}{C|\alpha|^2+A|\alpha_{2}|^2}\right)\left\|\left(\alpha_{1} K_1+\alpha_{2} K_2\right)^{\ast} f\right\|^2 \leq \int_{\Omega}\langle \Phi_{\omega} f,\Psi_{\omega} f \rangle d \mu(\omega) .
$$

Hence $(\Phi, \Psi)_{K}$ satisfies the lower frame condition. And we have
$$
\int_{\Omega}\langle \Phi_{\omega} f,\Psi_{\omega} f \rangle d \mu(\omega) \leq\min \lbrace B,D\rbrace\|f\|^2, \text { for all } f \in \mathcal{H}
$$
it follows that
$$
\left(\dfrac{A C}{C|\alpha|^2+A|\alpha_{2}|^2}\right)\left\|\left(\alpha_{1} K_1+\alpha_{2} K_2\right)^{\ast} f\right\|^2\leq\int_{\Omega}\langle \Phi_{\omega} f,\Psi_{\omega} f \rangle d \mu(\omega) \leq\min \lbrace B,D\rbrace\|f\|^2, \text { for all } f \in \mathcal{H}
$$
Therefore $(\Phi, \Psi)_{K}$ is continuous $(\alpha_{1} K_1+\alpha_{2} K_2)$-bi-$g$-frame.

(2) Now for each $f \in \mathcal{H}$, we have
$$
\left\|\left(K_1 K_2\right)^{\ast} f\right\|^2=\left\|K_2^{\ast} K_1^{\ast} f\right\|^2 \leq\left\|K_2^{\ast}\right\|^2\left\|K_1^{\ast} f\right\|^2 .
$$

Since  $(\Phi, \Psi)_{K}$ is continuous $K_1$-bi-$g$-frame, then there exist two constants $0<A \leq B<\infty$ such that
$$
A\left\|K_1^{\ast} f\right\|^2 \leq \int_{\Omega}\langle \Phi_{\omega} f,\Psi_{\omega} f \rangle d \mu(\omega) \leq B\|f\|^2, \text { for all } f \in \mathcal{H} .
$$
Therefore
$$
\dfrac{1}{\left\|K_2^{\ast}\right\|^2} \left\|\left(K_1 K_2\right)^{\ast} f\right\|^2\leq\left\|K_1^{\ast} f\right\|^2 \leq \dfrac{1}{A} \int_{\Omega}\langle \Phi_{\omega} f,\Psi_{\omega} f \rangle d \mu(\omega) \leq \dfrac{B}{A}\|f\|^2 .
$$

This implies that
$$
\dfrac{A}{\left\|K_2^{\ast}\right\|^2}\left\|\left(K_1 K_2\right)^{\ast} f\right\|^2 \leq \int_{\Omega}\langle \Phi_{\omega} f,\Psi_{\omega} f \rangle d \mu(\omega) \leq B\|f\|^2, \text { for all } f \in \mathcal{H} .
$$

Therefore $(\Phi, \Psi)_{K}$ is continuous $K_1 K_2$-bi-$g$-frame for $\mathcal{H} $.
\end{proof}

\begin{corollary}
 Let $n\in\mathbb{N}\setminus\lbrace  0,1\rbrace$ and $K_i \in \mathcal{B}(\mathcal{H})$ for $j\in  [\![1;n]\!] $. If  $(\Phi, \Psi)_{K}$ is continuous K$_{j}$-bi-$g$-frame for $j\in  [\![1;n]\!]$  and $\alpha_{1}, \alpha_{2}\cdots, \alpha_{n}$ are non-zero scalars. Then the following holds: 
\begin{enumerate}
\item $(\Phi, \Psi)_{K}$ is continuous $(\sum\limits_{j=1}^{n}\alpha_{j} K_i)$-bi-$g$-frame
\item $(\Phi, \Psi)_{K}$ is continuous $(K_1K_2\cdots K_{n})$-bi-$g$-frame.
\end{enumerate}
\end{corollary}
\begin{proof}
(1) Suppose that $n\in\mathbb{N}\setminus\lbrace  0,1\rbrace$ and for every $j\in  [\![1;n]\!]$, $(\Phi, \Psi)_{K}$ is continuous $K_{j}$-bi-$g$-frame . Then for each $j\in  [\![1;n]\!]$ there exist positive constants $0<A_{j} \leq B_{j}<\infty$ such that
$$
A_{j}\left\|K_i^{\ast} f\right\|^2 \leq \int_{\Omega}\langle \Phi_{\omega} f,\Psi_{\omega} f \rangle d \mu(\omega) \leq B_{j}\|f\|^2, \text { for all } f \in \mathcal{H} .
$$
Now, we can write
$$
\begin{aligned}
\left\|\left(\sum\limits_{j=1}^{n}\alpha_{j} K_i\right)^{\ast} f\right\|^2 & =\left\|\alpha_{1} K_1^{\ast} f+\left(\alpha_{2}K_2+\cdots+\alpha_{n} K_n\right)^{\ast} f\right\|^2 \\
& \leq|\alpha_{1}|^2\left\|K_1^{\ast} f\right\|^2+\left\|\left(\alpha_{2}K_2+\cdots+\alpha_{n} K_n\right)^{\ast} f\right\|^2 \\
& \leq|\alpha_{1}|^2\left\|K_1^{\ast} f\right\|^2+\cdots+|\alpha_{n}|^2\left\|K_n^{\ast} f\right\|^2 \\
& \leq|\alpha_{1}|^2\left(\dfrac{1}{A_{1}} \int_{\Omega}\langle \Phi_{\omega} f,\Psi_{\omega} f \rangle d \mu(\omega)\right)+\cdots+|\alpha_{n}|^2\left(\dfrac{1}{A_{n}} \int_{\Omega}\langle \Phi_{\omega} f,\Psi_{\omega} f \rangle d \mu(\omega)\right) \\
& =\left(\dfrac{|\alpha_{1}|^2}{A_{1}}+\cdots+\dfrac{|\alpha_{n}|^2}{A_{n}}\right) \int_{\Omega}\langle \Phi_{\omega} f,\Psi_{\omega} f \rangle d \mu(\omega)\\
&= \left(\sum_{j=1}^{n}\dfrac{|\alpha_{j}|^2}{A_{j}}\right) \int_{\Omega}\langle \Phi_{\omega} f,\Psi_{\omega} f \rangle d \mu(\omega) .
\end{aligned}
$$
Hence $(\Phi, \Psi)_{K}$ satisfies the lower frame condition. And we have
$$
\int_{\Omega}\langle \Phi_{\omega} f,\Psi_{\omega} f \rangle d \mu(\omega) \leq\min\limits_{j\in  [\![1;n]\!]} \lbrace B_{j}\rbrace\|f\|^2, \text { for all } f \in \mathcal{H}
.$$
It follows that
$$
\left(\sum_{j=1}^{n}\dfrac{|\alpha_{j}|^2}{A_{j}}\right)^{-1}\left\|\left(\sum\limits_{j=1}^{n}\alpha_{j} K_i\right)^{\ast} f\right\|^2\leq\int_{\Omega}\langle \Phi_{\omega} f,\Psi_{\omega} f \rangle d \mu(\omega) \leq\min\limits_{j\in  [\![1;n]\!]} \lbrace B_{j}\rbrace\|f\|^2, \text { for all } f \in \mathcal{H}.
$$
Hence $(\Phi, \Psi)_{K}$ is continuous $(\sum\limits_{j=1}^{n}\alpha_{j} K_i)$-bi-$g$-frame

(2) Now for each $f \in \mathcal{H}$, we have
$$
\left\|(K_1K_2\cdots K_{n})^{\ast} f\right\|^2=\left\|K_n^{\ast} \cdots K_1^{\ast} f\right\|^2 \leq\left\|K_n^{\ast} \cdots K_2^{\ast}\right\|^2\left\|K_1^{\ast} f\right\|^2 .
$$
Since  $(\Phi, \Psi)_{K}$ is continuous $K_1$-bi-$g$-frame, then there exist two constants $0<A_{1} \leq B_{1}<\infty$ such that 
$$
A_{1}\left\|K_1^{\ast} f\right\|^2 \leq \int_{\Omega}\langle \Phi_{\omega} f,\Psi_{\omega} f \rangle d \mu(\omega) \leq B_{1}\|f\|^2, \text { for all } f \in \mathcal{H} .
$$
Therefore
$$
\dfrac{1}{\left\|K_n^{\ast} \cdots K_2^{\ast}\right\|^2} \left\|(K_1K_2\cdots K_{n})^{\ast} f\right\|^2 \leq \dfrac{1}{A_{1}} \int_{\Omega}\langle \Phi_{\omega} f,\Psi_{\omega} f \rangle d \mu(\omega) \leq \dfrac{B_{1}}{A_{1}}\|f\|^2 .
$$

This implies that
$$
\dfrac{A_{1}}{\left\|K_n^{\ast} \cdots K_2^{\ast}\right\|^2}\left\|(K_1K_2\cdots K_{n})^{\ast} f\right\|^2\leq \int_{\Omega}\langle \Phi_{\omega} f,\Psi_{\omega} f \rangle d \mu(\omega) \leq B_{1}\|f\|^2, \text { for all } f \in \mathcal{H} .
$$

Therefore $(\Phi, \Psi)_{K}$ is continuous $(K_1K_2\cdots K_{n})$-bi-$g$-frame for $\mathcal{H} $.
\end{proof}
\begin{theorem}
Let $K \in \mathcal{B}(\mathcal{H})$ with $\|K\| \geq 1$. Then every ordinary continuous bi-$g$-frame is a continuous $K$-bi-$g$-frames for $\mathcal{H}$ with respect to $\left\{\mathcal{K}_{\omega}\right\}_{\in I} $.
\end{theorem}
\begin{proof}
Suppose that $(\Phi, \Psi)_{K}$
  is continuous bi-$g$-
  frame for $\mathcal{H}$. Then there exist two constants $0<A \leq B<\infty$ such that
$$
A\left\| f\right\|^2 \leq \int_{\Omega}\langle \Phi_{\omega} f,\Psi_{\omega} f \rangle d \mu(\omega) \leq B\|f\|^2, \text { for all } f \in \mathcal{H} .
$$
For $K \in \mathcal{B}(\mathcal{H})$, we have 
$$\left\|K^{\ast} f\right\|^{2} \leq\|K\|^{2}\|f\|^{2}, \quad \forall f \in \mathcal{H}.$$ Since $\|K\| \geq 1$, we obtain $$\dfrac{1}{\|K\|^2}\left\|K^{\ast} f\right\|^2 \leq\|f\|^2\quad \forall f \in \mathcal{H}.$$ Therefore
$$
\dfrac{A}{\|K\|^2}\left\|K^{\ast} f\right\|^2 \leq A\|f\|^2 \leq \int_{\Omega}\langle \Phi_{\omega} f,\Psi_{\omega} f \rangle d \mu(\omega) \leq B\|f\|^2, \text { for all } f \in \mathcal{H} \text {. }
$$

Therefore $(\Phi, \Psi)_{K}$ is a continuous $K$-bi-$g$-frames for $\mathcal{H}$.
\end{proof}
\begin{theorem}\label{Th3.7}
Let $(\Phi, \Psi)_{K}$ be a continuous bi-$g$-frame for $\mathcal{H}$. Then 
$ (\Phi, \Psi)_{K}$ is a continuous $K$-bi-$g$-frames for $\mathcal{H}$ with respect to $\left\{\mathcal{K}_{\omega}\right\}_{\in I} $ if and only if there exists $A>0$ such that $S_{\Phi, \Psi} \geq A K K^{\ast}$, where $S_{\Phi, \Psi}$ is the continuous bi-$g$-frame operator for $(\Phi, \Psi)_{K}$.
\end{theorem}

\begin{proof}
 $(\Phi, \Psi)_{K}$ is a continuous $K$-bi-$g$-frames for $\mathcal{H}$ with frame bounds $A, B$ and continuous bi-$g$-frame operator $S_{\Phi, \Psi} $, if and only if
$$
A\left\|K^{\ast} f\right\|^2 \leq \langle S_{\Phi, \Psi} f, f\rangle =\langle  \int_{\Omega} \Psi_{\omega}^{\ast} \Phi_{\omega} f d \mu(\omega), f\rangle = \int_{\Omega}\langle \Phi_{\omega} f,\Psi_{\omega} f \rangle d \mu(\omega)\leq B\|f\|^2, \quad \forall f \in \mathcal{H},
$$
that is,
$$
\left\langle A K K^{\ast} f, f\right\rangle \leq\langle S_{\Phi, \Psi} f, f\rangle \leq\langle B f, f\rangle, \quad \forall f \in \mathcal{H} .
$$

So the conclusion holds.
\end{proof}
\begin{corollary}
Let $(\Phi, \Psi)_{K}$ be a bi-$g$-frame for $\mathcal{H}$. Then 
$ (\Phi, \Psi)_{K}$ is a tight continuous $K$-bi-$g$-frames for $\mathcal{H}$ with respect to $\left\{\mathcal{K}_{\omega}\right\}_{\in I} $ if and only if there exists $A>0$ such that $S_{\Phi, \Psi} = A K K^{\ast}$, where $S_{\Phi, \Psi}$ is the continuous bi-$g$-frame operator for $(\Phi, \Psi)_{K}$.
\end{corollary}
\begin{proof}
The proof is evident; one can simply utilize the definition of tight continuous $K$-bi-$g$-frames \ref{Def Tight}.
\end{proof}
\begin{theorem}
 Let $(\Phi, \Psi)_{K}$ be a continuous bi-$g$-frame for $\mathcal{H}$, with continuous bi-$g$-frame operator $S_{\Phi, \Psi}$ which satisfies $S_{\Phi, \Psi}^{\frac{1}{2}^{\ast}}=S_{\Phi, \Psi}^{\frac{1}{2}}$. Then $(\Phi, \Psi)_{K}$ is a continuous $K$-bi-$g$-frame for $\mathcal{H}$ with respect to $\left\{\mathcal{K}_{\omega}\right\}_{\in I} $ if and only if $K=S_{\Phi, \Psi}^{\frac{1}{2}} U$, for some $U \in \mathcal{B}(\mathcal{H})$.
\end{theorem}
\begin{proof}
Assume that  $(\Phi, \Psi)_{K}$ is a continuous $K$-bi-$g$-frames, by Theorem \ref{Th3.7}, there exists $A>0$ such that
$$
A K K^{\ast} \leq S_{\Phi, \Psi}^{\frac{1}{2}} S_{\Phi, \Psi}^{\frac{1}{2}^{\ast}} .
$$

Then for each $f \in \mathcal{H}$,
$$\left\|K^{\ast} f\right\|^2 \leq \lambda^{-1}\left\|S_{\Phi, \Psi}^{\frac{1}{2}^{\ast}} f\right\|^2.$$ Therefore Theorem \ref{Doglas th}, $K=S_{\Phi, \Psi}^{\frac{1}{2}} U$, for some $U \in \mathcal{B}(\mathcal{H})$.

Conversely, let $K=S_{\Phi, \Psi}^{\frac{1}{2}} W$, for some $W \in \mathcal{B}(\mathcal{H})$. Then by Theorem \ref{Doglas th}, there is a positive number $\mu$ such that
$$
\left\|K^{\ast} f\right\| \leq \mu\left\|S_{\Phi, \Psi}^{\frac{1}{2}} f\right\|, \text { for all } f \in \mathcal{H}
$$
which implies that $$
\mu K K^{\ast} \leq S_{\Phi, \Psi}^{\frac{1}{2}} S_{\Phi, \Psi}^{\frac{1}{2}^{\ast}} .
$$ Since $S_{\Phi, \Psi}^{\frac{1}{2}^{\ast}}=S_{\Phi, \Psi}^{\frac{1}{2}}$ Then by Theorem \ref{Doglas th}, $(\Phi, \Psi)_{K}$ is a continuous $K$-bi-$g$-frames for $\mathcal{H}$.
\end{proof}
\section{Operators on continuous $K$-bi-$g$-frames in Hilbert Spaces}\label{s7}
In the following proposition we will require a necessary condition for the operator $\mathcal{T}$ for which $ (\Phi, \Psi)_{K}$  will be continuous $\mathcal{T}$-bi-$g$-frame for $\mathcal{H}$ with respect to $\left\{\mathcal{K}_{\omega}\right\}_{\in I} $.
 \begin{proposition}
   Let $(\Phi, \Psi)_{K}$ be a continuous $K$-bi-$g$-frames for $\mathcal{H}$. Let $\mathcal{T} \in \mathcal{B}(\mathcal{H})$ with $R(\mathcal{T}) \subseteq$ $\mathcal{R}(K)$. Then $(\Phi, \Psi)_{K}$ is a continuous $\mathcal{T}$-bi-$g$-frame for $\mathcal{H}$.
\end{proposition}
\begin{proof}
Suppose that $(\Phi, \Psi)_{K}$ is a continuous $K$-bi-$g$-frames for $\mathcal{H}$. Then there are positive constants $0<A \leq B<\infty$ such that
$$
A\left\|K^* f\right\|^2  \leq \int_{\Omega}\langle \Phi_{\omega} f,\Psi_{\omega} f \rangle d \mu(\omega) \leq B\|f\|^2, \text { for all } f \in \mathcal{H} .
$$

Since $R(\mathcal{T}) \subseteq \mathcal{R}(K)$, by Theorem \ref{Doglas th}, there exists $\alpha>0$ such that $$\mathcal{T} \mathcal{T}^{\ast} \leq \alpha^2 K K^*.$$

Hence, 
$$
\dfrac{A}{\alpha^2}\left\|\mathcal{T}^{\ast} f\right\|^2 \leq A\left\|K^* f\right\|^2 \leq  \int_{\Omega}\langle \Phi_{\omega} f,\Psi_{\omega} f \rangle d \mu(\omega) \leq B\|f\|^2, \text { for all } f \in \mathcal{H} .
$$

Hence $(\Phi, \Psi)_{K}$ is a continuous $\mathcal{T}$-bi-$g$-frame for $\mathcal{H}$.
\end{proof}

\begin{theorem}
 Let $(\Phi, \Psi)_{K}$ be a continuous $K$-bi-$g$-frames for $\mathcal{H}$ with continuous bi-$g$-frame operator $S_{\Phi, \Psi}$ and let $\mathcal{T}$ be a positive operator. Then $(\Phi+\mathcal{T}\Phi, \Psi+\mathcal{T}\Psi)_{K}=\left(\left\{\Phi_{\omega}+\mathcal{T}\Phi_{\omega}\right\}_{\omega \in \Omega},\left\{\Psi_{\omega}+\mathcal{T}\Psi_{\omega}\right\}_{\omega \in \Omega}\right)$ is a continuous $K$-bi-$g$-frames. 
 
 Moreover for any $n\in\mathbb{N}^{\ast}$, $\left(\left\{\Phi_{\omega}+\mathcal{T}^{n}\Phi_{\omega}\right\}_{\omega \in \Omega},\left\{\Psi_{\omega}+\mathcal{T}^{n}\Psi_{\omega}\right\}_{\omega \in \Omega}\right)$ is a continuous $K$-bi-$g$-frames for $\mathcal{H}$.
\end{theorem}
\begin{proof}
Suppose that $(\Phi, \Psi)_{K}$ is a continuous $K$-bi-$g$-frames for $\mathcal{H}$. Then by Theorem \ref{Th3.7}, there exists $m>0$ such that $S_{\Phi, \Psi} \geq m K K^*$. For every $f \in \mathcal{H}$, we have
$$
\begin{aligned}
S_{(\Phi+\mathcal{T}\Phi),( \Psi+\mathcal{T}\Psi)}f&=\int_{\Omega}\left(\Psi_{\omega}+\mathcal{T} \Psi_{\omega}\right)^{\ast}\left(\Phi_{\omega}+\mathcal{T} \Phi_{\omega}\right)f  d \mu(\omega)\\
 & =(I+\mathcal{T})^{\ast} \left[\int_{\Omega}\Psi_{\omega}^{\ast}\Phi_{\omega} f   d \mu(\omega)\right]  (I+\mathcal{T})  \\
& =(I+\mathcal{T})^{\ast} S_{\Phi,\Psi}f(I+\mathcal{T}) .
\end{aligned}
$$
Hence the frame operator for $(\Phi+\mathcal{T}\Phi, \Psi+\mathcal{T}\Psi)_{K}$ is $(I+\mathcal{T})^{\ast} S_{\Phi,\Psi}(I+\mathcal{T})$.  
Since $\mathcal{T}$ is positive operator we get,
$$
(I+\mathcal{T})^{\ast} S_{\Phi,\Psi}(I+\mathcal{T})= S_{\Phi,\Psi}+ S_{\Phi,\Psi} \mathcal{T}+\mathcal{T}^{\ast}  S_{\Phi,\Psi}+\mathcal{T}^{\ast}  S_{\Phi,\Psi} \mathcal{T} \geq  S_{\Phi,\Psi} \geq m K K^*,
$$
Once again, applying Theorem \ref{Th3.7}, we can conclude that $(\Phi+\mathcal{T}\Phi, \Psi+\mathcal{T}\Psi)_{K}$ is a continuous $K$-bi-$g$-frames for $\mathcal{H}$.

Now, for any $n\in\mathbb{N}^{\ast}$, the frame operator for
 $$ S_{(\Phi+\mathcal{T}^{n}\Phi),( \Psi+\mathcal{T}^{n}\Psi)}=\left(I+\mathcal{T}^{n}\right)^{\ast} S_{\Phi,\Psi}(I+\left.\mathcal{T}^{n}\right) \geq S_{\Phi,\Psi}. $$ Hence $\left(\left\{\Phi_{\omega}+\mathcal{T}^{n}\Phi_{\omega}\right\}_{\omega \in \Omega},\left\{\Psi_{\omega}+\mathcal{T}^{n}\Psi_{\omega}\right\}_{\omega \in \Omega}\right)$ is a continuous $K$-bi-$g$-frames for $\mathcal{H}$.
\end{proof}

\begin{theorem}
Let $K \in \mathcal{B}(\mathcal{H})$ and $(\Phi, \Psi)_{K}$ be a continuous $K$-bi-$g$-frames for $\mathcal{H}$ with respect to $\left\{\mathcal{K}_{\omega}\right\}_{\omega \in \Omega}$, and that $M \in \mathcal{B}(\mathcal{H})$ has closed range with $M K=K M$. If $\mathcal{R}\left(K^*\right) \subset \mathcal{R}(M)$, then $(\Phi M^*, \Psi M^*)_{K}=\left(\left\{\Phi_{\omega} M^*\right\}_{\omega \in \Omega}, \left\{\Psi_{\omega} M^*\right\}_{\omega \in \Omega}\right)$ is a continuous $K$-bi-$g$-frame for $\mathcal{H}$ with respect to $\left\{\mathcal{K}_{\omega}\right\}_{\omega \in \Omega}$.
\end{theorem}
\begin{proof}
 For every $f \in \mathcal{H}$, we have
 $$
A\left\|K^* f\right\|^2  \leq \int_{\Omega}\langle \Phi_{\omega} f,\Psi_{\omega} f \rangle d \mu(\omega) \leq B\|f\|^2 .
$$
Then for $M \in \mathcal{B}(\mathcal{H})$, we get
$$
\int_{\Omega}\langle \Phi_{\omega} M^* f,\Psi_{\omega} M^* f \rangle  d \mu(\omega)\leq B\left\|M^* f\right\|^2 \leq B\|M\|^2\|f\|^2 .
$$
Since $M$ has closed range and $\mathcal{R}\left(K^*\right) \subset \mathcal{R}(M)$,
$$
\begin{aligned}
\left\|K^* f\right\|^2 & =\left\|M M^{+} K^* f\right\|^2\\ &=\left\|\left(M^{+}\right)^* M^* K^* f\right\|^2 \\
& =\left\|\left(M^{+}\right)^* K^* M^* f\right\|^2\\  &\leq\left\|M^{+}\right\|^2\left\|K^* M^* f\right\|^2.
\end{aligned}
$$
On the other hand, we have
$$
\int_{\Omega}\langle \Phi_{\omega} M^* f,\Psi_{\omega} M^* f \rangle  d \mu(\omega)\geq A \left\|K^* M^* f\right\|^2 \geq A\left\|M^{+}\right\|^{-2}\left\|K^* f\right\|^2
$$
Hence $(\Phi M^*, \Psi M^*)_{K}$ is a continuous $K$-bi-$g$-frame for $\mathcal{H}$ with respect to $\left\{\mathcal{K}_{\omega}\right\}_{\omega \in \Omega}$.
\end{proof}
\begin{theorem}
Let $K, M \in B(\mathcal{H})$ and $(\Phi, \Psi)_{K}$ be a continuous $\delta$-tight $K$-$g$-frame for $\mathcal{H}$ with respect to $\left\{\mathcal{K}_{\omega}\right\}_{\omega \in \Omega}$. If $\mathcal{R}\left(K^*\right)=\mathcal{H}$ and $M K=K M$, then $(\Phi M^*, \Psi M^*)_{K}=\left(\left\{\Phi_{\omega} M^*\right\}_{\omega \in \Omega}, \left\{\Psi_{\omega} M^*\right\}_{\omega \in \Omega}\right)$ is a continuous $K$-bi-$g$-frame for $\mathcal{H}$ with respect to $\left\{\mathcal{K}_{\omega}\right\}_{\omega \in \Omega}$ if and only if $M$ is surjective.
\end{theorem}
\begin{proof}
Suppose that $\left(\left\{\Phi_{\omega} M^*\right\}_{\omega \in \Omega}, \left\{\Psi_{\omega} M^*\right\}_{\omega \in \Omega}\right)$ is a continuous $K$-bi-$g$-frame for $\mathcal{H}$ with respect to $\left\{\mathcal{K}_{\omega}\right\}_{\omega \in \Omega}$ with frame bounds $A$ and $B$. that is for every $f\in\mathcal{H}$,
$$
A\left\|K^* f\right\|^2 \leq \int_{\Omega}\langle \Phi_{\omega} M^* f,\Psi_{\omega} M^* f \rangle  d \mu(\omega)\leq B\|f\|^2 .
$$

and we have $$
A\left\|K^* f\right\|^2=\leq \int_{\Omega}\langle \Phi_{\omega} f,\Psi_{\omega} f \rangle d \mu(\omega), \text { for all } f \in \mathcal{H} .
$$ Since $K^* M^*=M^* K^*$, we obtain
$$
 \delta\left\|M^* K^* f\right\|^2= \delta\left\|K^* M^* f\right\|^2=\int_{\Omega}\langle \Phi_{\omega} M^* f,\Psi_{\omega} M^* f \rangle d \mu(\omega) .
$$
Hence
$$
\left\|M^* K^* f\right\|^2=\dfrac{1}{ \delta} \int_{\Omega}\langle \Phi_{\omega} M^* f,\Psi_{\omega} M^* f \rangle  d \mu(\omega)\geq \dfrac{A}{ \delta}\left\|K^* f\right\|^2.
$$
from which we conclude that $M^*$ is injective since $\mathcal{R}\left(K^*\right)=\mathcal{H}$, $M$ is surjective as a consequence.
\end{proof}
\section{Stability of continuous $K$-bi-$g$-frames for Hilbert spaces}\label{s8}
\begin{theorem} \label{th51}
 Suppose that $K \in \mathcal{B}(\mathcal{H})$ and $K$ has closed range. Let $\Phi=\left\{\Phi_{\omega}\;:\;\Phi_{\omega}\in\mathcal{B}(\mathcal{H},\mathcal{K}_{\omega})\right\}_{\omega \in \Omega}$ and $\Psi=\left\{\Psi_{\omega}\;:\;\Psi_{\omega}\in\mathcal{B}(\mathcal{H},\mathcal{K}_{\omega})\right\}_{\omega \in \Omega}$ are two continuous $g$-Bessel sequences with bounds $B_{\Phi}$, $B_{\Psi}$ respectively. Assume that $(\Phi, \Psi)_{K}$ be a continuous $K$-bi-$g$-frame for $\mathcal{H}$ with respect to $\left\{\mathcal{K}_{\omega}\right\}_{\omega \in \Omega}$ with bounds $A$ and $B$ and $\left(\left\{\Lambda_{\omega}\right\}_{\omega \in \Omega},\left\{\Gamma_{\omega}\right\}_{\omega \in \Omega}\right)$ be a pair of sequences for $\mathcal{H}$ with respect to $\left\{\mathcal{K}_{\omega}\right\}_{\omega \in \Omega}$. If there exist constants $\alpha, \beta, \gamma \in[0,1)$ such that $\max \left\{\alpha+\gamma , \beta\right\}<1$ and
$$
\begin{aligned}
\left\|\int_{\Omega}\left(\Psi_{\omega}^* \Phi_{\omega}-\Gamma_{\omega}^* \Lambda_{\omega}\right) f d \mu(\omega)\right\| 
\leq & \alpha\left\|\int_{\Omega} \Psi_{\omega}^* \Phi_{\omega} f d \mu(\omega)\right\|+\beta\left\|\int_{\Omega} \Gamma_{\omega}^* \Lambda_{\omega} f d \mu(\omega)\right\|+\gamma\left\| f\right\|.
\end{aligned}
$$
Then    $\left(\left\{\Lambda_{\omega}\right\}_{\omega \in \Omega},\left\{\Gamma_{\omega}\right\}_{\omega \in \Omega}\right)$ is a continuous $K$-bi-$g$-frame for $\mathcal{H}$ with respect to $\left\{\mathcal{K}_{\omega}\right\}_{\omega \in \Omega}$ with bounds
$$
A \dfrac{\left[1-\left( \alpha+\gamma \right)\right]}{\left(1+\beta\right)}, \dfrac{\left(1+\alpha\right) \sqrt{B_{\Phi}B_{\Psi}}+\gamma}{1-\beta}.
$$
\end{theorem}
\begin{proof}
Suppose that $ J\subset I ,|J|<+\infty$. For any $f \in \mathcal{H}$, we have
$$
\begin{aligned}
\left\|\int_{\Omega} \Gamma_{\omega}^* \Lambda_{\omega}  f d \mu(\omega)\right\| & \leq\left\|\int_{\Omega}\left(\Gamma_{\omega}^* \Lambda_{\omega}-\Psi_{\omega}^* \Phi_{\omega}\right)  f d \mu(\omega)\right\|+\left\|\int_{\Omega} \Psi_{\omega}^* \Phi_{\omega}  f d \mu(\omega)\right\| \\
& \leq\left(1+\alpha\right)\left\|\int_{\Omega} \Psi_{\omega}^* \Phi_{\omega}  f d \mu(\omega)\right\|+\beta\left\|\int_{\Omega} \Gamma_{\omega}^* \Lambda_{\omega}  f d \mu(\omega)\right\|+\gamma \left\| f\right\| .
\end{aligned}
$$
Then
$$
\left\|\int_{\Omega} \Gamma_{\omega}^* \Lambda_{\omega}  f d \mu(\omega)\right\| \leq \dfrac{1+\alpha}{1-\beta}\left\|\int_{\Omega} \Psi_{\omega}^* \Phi_{\omega}  f d \mu(\omega)\right\|+\dfrac{\gamma}{1-\beta}\left\| f\right\| .
$$

Since
$$
\begin{aligned}
\left\|\int_{\Omega} \Psi_{\omega}^* \Phi_{\omega}  f d \mu(\omega)\right\| & =\sup _{\|y\|=1}\left|\left\langle\int_{\Omega} \Psi_{\omega}^* \Phi_{\omega}  f d \mu(\omega), y\right\rangle\right| \\ & =\sup _{\|y\|=1}\left|\int_{\Omega}\left\langle \Psi_{\omega}^* \Phi_{\omega}  f , y\right\rangle d \mu(\omega)\right| \\ &=\sup _{\|y\|=1}\left|\int_{\Omega}\left\langle \Phi_{\omega} f, \Psi_{\omega} y\right\rangle d \mu(\omega)\right|
\\ 
& \leq\left(\int_{\Omega}\left\|\Phi_{\omega} f\right\|^2  d \mu(\omega)\right)^{\dfrac{1}{2}} \sup _{\|y\|=1}\left(\int_{\Omega}\left\|\Psi_{\omega} y\right\|^2  d \mu(\omega)\right)^{\dfrac{1}{2}} \\
& \leq \sqrt{B_{\Phi}B_{\Psi}} \left\| f\right\|.
\end{aligned}
$$
 Hence, for all $f \in \mathcal{H}$, we have
$$
\left\|\int_{\Omega} \Gamma_{\omega}^* \Lambda_{\omega}  f d \mu(\omega)\right\| \leq \dfrac{\left(1+\alpha\right) \sqrt{B_{\Phi}B_{\Psi}}}{1-\beta}\left\| f\right\| +\dfrac{\gamma}{1-\beta}\left\| f\right\|= \dfrac{\left(1+\alpha\right) \sqrt{B_{\Phi}B_{\Psi}}+\gamma}{1-\beta} \|f\| .
$$
Thus $\int_{\Omega} \Gamma_{\omega}^* \Lambda_{\omega}  f d \mu(\omega)$ is unconditionally convergent. we considere
$$
\mathcal{M} :\mathcal{H} \rightarrow \mathcal{H}, \quad \mathcal{M} f=\int_{\Omega} \Gamma_{\omega}^* \Lambda_{\omega}  f d \mu(\omega).
$$

Then $\mathcal{M}$ is well-defined, bounded and $$\|\mathcal{M}\| \leq \dfrac{\left(1+\alpha\right) \sqrt{B_{\Phi}B_{\Psi}}+\gamma}{1-\beta}.$$ For every $f \in \mathcal{H}$, we have 
\begin{equation} \label{eq51}
\langle \mathcal{M} f, f\rangle=\langle\int_{\Omega}\Gamma_{\omega}^{\ast} \Lambda_{\omega}  f d \mu(\omega),f\rangle=\int_{\Omega}\langle \Lambda_{\omega} f,\Gamma_{\omega} f\rangle d \mu(\omega) \leq\|\mathcal{M}\|\|f\|^2
\end{equation} 
It implies that $\left(\left\{\Lambda_{\omega}\right\}_{\omega \in \Omega},\left\{\Gamma_{\omega}\right\}_{\omega \in \Omega}\right)$ is a continuous bi-$g$-Bessel sequence for $\mathcal{H}$ with respect to $\left\{\mathcal{K}_{\omega}\right\}_{\omega \in \Omega}$.
Let $S_{\Phi, \Psi}$ be the bi-$g$-frame operator of $(\Phi, \Psi)_{K}$. According to the theorem hypothesis, we obtain
$$
\|(S_{\Phi, \Psi}-\mathcal{M}) f\| \leq \alpha\|S_{\Phi, \Psi} f\|+\beta\|\mathcal{M} f\|+\gamma\|f\|, \forall f \in \mathcal{H} .
$$

Then,
$$
\begin{aligned}
\left\|f-\mathcal{M} S_{\Phi, \Psi}^{-1} f\right\| & \leq \alpha\|f\|+\beta\left\|\mathcal{M} S_{\Phi, \Psi}^{-1} f\right\|+\gamma\|f\| \\
& \leq\left(\alpha+\gamma\right)\|f\|+\beta\left\|\mathcal{M} S_{\Phi, \Psi}^{-1} f\right\|
\end{aligned}
$$
Since $0 \leq \max \left\{\alpha+\gamma, \beta\right\}<1$, According to Lemma \ref{Lemma2.5} , we get
$$
\dfrac{1-\beta}{1+\left(\alpha+\gamma\right)} \leq\left\|S_{\Phi, \Psi} \mathcal{M}^{-1}\right\| \leq \dfrac{1+\beta}{1-\left(\alpha+\gamma\right)} .
$$
Since
$$\|S_{\Phi, \Psi}\|=\|S_{\Phi, \Psi}\mathcal{M}^{-1}\mathcal{M}\|\leq \|S_{\Phi, \Psi}\mathcal{M}^{-1}\|\|\mathcal{M}\| $$
Therefore, 
\begin{equation}\label{eq52}
\|\mathcal{M}\| \geq \dfrac{A}{\left\|S_{\Phi, \Psi} \mathcal{M}^{-1}\right\|}\left\|K K^*\right\| \geq A \dfrac{\left[1-\left( \alpha+\gamma \right)\right]}{\left(1+\beta\right)}\left\|K K^*\right\| .
\end{equation}
Hence, by Theorem \ref{Th3.7}, we can conclude that $\left(\left\{\Lambda_{\omega}\right\}_{\omega \in \Omega},\left\{\Gamma_{\omega}\right\}_{\omega \in \Omega}\right)$ is a continuous $K$-bi-$g$-frame for $\mathcal{H}$ with respect to $\left\{\mathcal{K}_{\omega}\right\}_{\omega \in \Omega}$.
\end{proof}
\begin{corollary}
Suppose that $K \in \mathcal{B}(\mathcal{H})$ and $K$ has closed range. Let $\Phi=\left\{\Phi_{\omega}\;:\;\Phi_{\omega}\in\mathcal{B}(\mathcal{H},\mathcal{K}_{\omega})\right\}_{\omega \in \Omega}$ and $\Psi=\left\{\Psi_{\omega}\;:\;\Psi_{\omega}\in\mathcal{B}(\mathcal{H},\mathcal{K}_{\omega})\right\}_{\omega \in \Omega}$ are two continuous $g$-Bessel sequences with bounds $B_{\Phi}$, $B_{\Psi}$ respectively. Assume that $(\Phi, \Psi)_{K}$ be a continuous $K$-bi-$g$-frame for $\mathcal{H}$ with respect to $\left\{\mathcal{K}_{\omega}\right\}_{\omega \in \Omega}$ with bounds $A$ and $B$ and $\left(\left\{\Lambda_{\omega}\right\}_{\omega \in \Omega},\left\{\Gamma_{\omega}\right\}_{\omega \in \Omega}\right)$ be a pair of sequences for $\mathcal{H}$ with respect to $\left\{\mathcal{K}_{\omega}\right\}_{\omega \in \Omega}$.  If there exists constant $0<D<A$ such that
$$
\left\|\int_{\Omega} \left(\Psi_{\omega}^* \Phi_{\omega} -\Gamma_{\omega}^* \Lambda_{\omega}\right) f d \mu(\omega)\right\| \leq D\left\|K^* f\right\|, \forall f \in \mathcal{H},
$$
then $\left(\left\{\Lambda_{\omega}\right\}_{\omega \in \Omega},\left\{\Gamma_{\omega}\right\}_{\omega \in \Omega}\right)$ is a continuous $K$-bi-$g$-frame for $\mathcal{H}$ with respect to $\left\{\mathcal{K}_j\right\}_{j \in J}$ with bounds $A\left(1-D\sqrt{\dfrac{B}{A}} \right)$ and $\left(\sqrt{B_{\Phi}B_{\Psi}}+D\sqrt{\dfrac{B}{A}}\right)$.
\end{corollary}
\begin{proof} For any $f \in \mathcal{H}$, we have  $$\left\|K^* f\right\| \leq  \dfrac{1}{\sqrt{A}}\left(\leq \int_{\Omega}\langle \Phi_{\omega} f,\Psi_{\omega} f \rangle d \mu(\omega)\right) ^{\dfrac{1}{2}}$$
It is clear that $\int_{\Omega} \Gamma_{\omega}^* \Lambda_{\omega} f$ is convergent for any $f \in \mathcal{H}$. Then,
$$
\begin{aligned}
\left\|\int_{\Omega}\left(\Psi_{\omega}^* \Phi_{\omega} -\Gamma_{\omega}^* \Lambda_{\omega}\right) f d \mu(\omega)\right\| &\leq D\left\|K^* f\right\|\\ &\leq  \dfrac{1}{\sqrt{A}}\left(\leq \int_{\Omega}\langle \Phi_{\omega} f,\Psi_{\omega} f \rangle d \mu(\omega)\right) ^{\dfrac{1}{2}}\\
&\leq  D\sqrt{\dfrac{B}{A}}\|f\|.
\end{aligned}
$$

By letting $\alpha, \beta=0, \gamma= D\sqrt{\dfrac{B}{A}}$ in Theorem \ref{th51}, $\left(\left\{\Lambda_{\omega}\right\}_{\omega \in \Omega},\left\{\Gamma_{\omega}\right\}_{\omega \in \Omega}\right)$ is a continuous $K$-bi-$g$-frame for $\mathcal{H}$ with respect to $\left\{\mathcal{K}_{\omega}\right\}_{\omega \in \Omega}$ with bounds $A\left(1-D\sqrt{\dfrac{B}{A}} \right)$ and $\left(\sqrt{B_{\Phi}B_{\Psi}}+D\sqrt{\dfrac{B}{A}}\right)$.
\end{proof}
\begin{theorem} \label{th53}
 Suppose that $K \in \mathcal{B}(\mathcal{H})$ and $K$ has closed range. Let $\Phi=\left\{\Phi_{\omega}\;:\;\Phi_{\omega}\in\mathcal{B}(\mathcal{H},\mathcal{K}_{\omega})\right\}_{\omega \in \Omega}$ and $\Psi=\left\{\Psi_{\omega}\;:\;\Psi_{\omega}\in\mathcal{B}(\mathcal{H},\mathcal{K}_{\omega})\right\}_{\omega \in \Omega}$ are two continuous $g$-Bessel sequences with bounds $B_{\Phi}$, $B_{\Psi}$ respectively. Assume that $(\Phi, \Psi)_{K}$ be a continuous $K$-bi-$g$-frame for $\mathcal{H}$ with respect to $\left\{\mathcal{K}_{\omega}\right\}_{\omega \in \Omega}$ with bounds $A$ and $B$ and $\left(\left\{\Lambda_{\omega}\right\}_{\omega \in \Omega},\left\{\Gamma_{\omega}\right\}_{\omega \in \Omega}\right)$ be a pair of sequences for $\mathcal{H}$ with respect to $\left\{\mathcal{K}_{\omega}\right\}_{\omega \in \Omega}$. If there exist constants $\alpha, \beta, \gamma \in[0,1)$ such that $\max \left\{\alpha+\gamma\sqrt{\dfrac{ B }{A}} , \beta\right\}<1$ and
$$
\begin{aligned}
\left\|\int_{\Omega}\left(\Psi_{\omega}^* \Phi_{\omega}-\Gamma_{\omega}^* \Lambda_{\omega}\right) f\right\| 
\leq & \alpha\left\|\int_{\Omega} \Psi_{\omega}^* \Phi_{\omega} f\right\|+\beta\left\|\int_{\Omega} \Gamma_{\omega}^* \Lambda_{\omega} f\right\|+\gamma\left\|K^{\ast} f\right\|.
\end{aligned}
$$
Then    $\left(\left\{\Lambda_{\omega}\right\}_{\omega \in \Omega},\left\{\Gamma_{\omega}\right\}_{\omega \in \Omega}\right)$ is a continuous $K$-bi-$g$-frame for $\mathcal{H}$ with respect to $\left\{\mathcal{K}_{\omega}\right\}_{\omega \in \Omega}$ with bounds
$$
A \dfrac{\left[1-\left( \alpha+\gamma\sqrt{\dfrac{ B }{A}}  \right)\right]}{\left(1+\beta\right)}, \dfrac{\left[\left(1+\alpha\right) \sqrt{B_{\Phi}B_{\Psi}}+\gamma\sqrt{\dfrac{ B }{A}}\right] }{1-\beta}.
$$
\end{theorem}
\begin{proof}
The proof is analogous to that of Theorem \ref{th51}.
\end{proof}
\begin{theorem} \label{th54}
 Suppose that $K \in \mathcal{B}(\mathcal{H})$ and $K$ has closed range. Let $\Phi=\left\{\Phi_{\omega}\;:\;\Phi_{\omega}\in\mathcal{B}(\mathcal{H},\mathcal{K}_{\omega})\right\}_{\omega \in \Omega}$ and $\Psi=\left\{\Psi_{\omega}\;:\;\Psi_{\omega}\in\mathcal{B}(\mathcal{H},\mathcal{K}_{\omega})\right\}_{\omega \in \Omega}$ are two continuous $g$-Bessel sequences with bounds $B_{\Phi}$, $B_{\Psi}$ respectively. Assume that $(\Phi, \Psi)_{K}$ be a continuous $K$-bi-$g$-frame for $\mathcal{H}$ with respect to $\left\{\mathcal{K}_{\omega}\right\}_{\omega \in \Omega}$ with bounds $A$ and $B$ and $\left(\left\{\Lambda_{\omega}\right\}_{\omega \in \Omega},\left\{\Gamma_{\omega}\right\}_{\omega \in \Omega}\right)$ be a pair of sequences for $\mathcal{H}$ with respect to $\left\{\mathcal{K}_{\omega}\right\}_{\omega \in \Omega}$. If there exist constants $\alpha, \beta, \sigma, \gamma \in[0,1)$ such that $\max \left\{\alpha+\sigma+\gamma\sqrt{\dfrac{ B }{A}} , \beta\right\}<1$ and
$$
\begin{aligned}
\left\|\int_{\Omega}\left(\Psi_{\omega}^* \Phi_{\omega}-\Gamma_{\omega}^* \Lambda_{\omega}\right) f\right\| 
\leq & \alpha\left\|\int_{\Omega} \Psi_{\omega}^* \Phi_{\omega} f\right\|+\beta\left\|\int_{\Omega} \Gamma_{\omega}^* \Lambda_{\omega} f\right\|+\sigma\left\| f\right\|+\gamma\left\|K^{\ast} f\right\|.
\end{aligned}
$$
Then $\left(\left\{\Lambda_{\omega}\right\}_{\omega \in \Omega},\left\{\Gamma_{\omega}\right\}_{\omega \in \Omega}\right)$ is a continuous $K$-bi-$g$-frame for $\mathcal{H}$ with respect to $\left\{\mathcal{K}_{\omega}\right\}_{\omega \in \Omega}$ with bounds
$$
A \dfrac{\left[1-\left( \alpha+\sigma+\gamma\sqrt{\dfrac{ B }{A}}  \right)\right]}{\left(1+\beta\right)}, \dfrac{\left[\left(1+\alpha\right) \sqrt{B_{\Phi}B_{\Psi}}+\sigma+\gamma\sqrt{\dfrac{ B }{A}}\right] }{1-\beta}.
$$
\end{theorem}
\begin{proof}
The proof is similar to that of Theorem \ref{th51}.
\end{proof}
\section*{Declarations}

\medskip

\noindent \textbf{Availablity of data and materials}\newline
\noindent Not applicable.

\medskip

\noindent \textbf{Human and animal rights}\newline
\noindent We would like to mention that this article does not contain any studies
with animals and does not involve any studies over human being.

\medskip

\noindent \textbf{Conflict of interest}\newline
\noindent The authors declare that they have no competing interests.

\medskip

\noindent \textbf{Fundings} \newline
\noindent The authors declare that there is no funding available for this paper.

\medskip

\noindent \textbf{Authors' contributions}\newline
\noindent The authors equally conceived of the study, participated in its
design and coordination, drafted the manuscript, participated in the
sequence alignment, and read and approved the final manuscript.

\end{document}